\documentclass[11pt,english]{article}

\usepackage[margin = 1.9 cm, bottom=16 mm,footskip=7mm,top=18mm]{geometry}

\usepackage{amsthm}
\usepackage{amsmath}
\usepackage{amssymb}
\usepackage{setspace}
\usepackage{mathtools}
\usepackage{graphicx}
\usepackage[hidelinks]{hyperref}
\usepackage{thm-restate}
\usepackage{cleveref}
\usepackage{hyperref}
\usepackage{enumitem}
\usepackage{framed}
\usepackage{subcaption}

\usepackage[square,sort,comma,numbers]{natbib}
\usepackage{floatrow}

\theoremstyle{plain}

\newtheorem*{thm*}{Theorem}
\newtheorem{thm}{Theorem}
\Crefname{thm}{Theorem}{Theorems}

\newtheorem*{lem*}{Lemma}
\newtheorem{lem}[thm]{Lemma}
\Crefname{lem}{Lemma}{Lemmas}

\newtheorem*{claim*}{Claim}

\crefname{claim}{Claim}{Claims}
\Crefname{claim}{Claim}{Claims}

\Crefname{prop}{Proposition}{Propositions}

\crefname{cor}{Corollary}{Corollaries}

\crefname{conj}{Conjecture}{Conjectures}

\Crefname{qn}{Question}{Questions}

\Crefname{obs}{Observation}{Observations}

\Crefname{ex}{Example}{Examples}

\theoremstyle{definition}

\Crefname{prob}{Problem}{Problems}

\Crefname{defn}{Definition}{Definitions}

\theoremstyle{remark}

\renewenvironment{proof}[1][]{\begin{trivlist}
\item[\hspace{\labelsep}{\bf\noindent Proof#1.\/}] }{\qed\end{trivlist}}

\newcommand{\remove}[1]{}

\newcommand{\floor}[1]{
    \lfloor #1 \rfloor
}

\newcommand{\eps}{\varepsilon}

\begin{document}


\title{Counting odd cycle free orientations of graphs}
\date{}
\author{
Matija Buci\'c\thanks{Department of Mathematics, ETH, Z\"urich, Switzerland. Email: \href{mailto:matija.bucic@math.ethz.ch} {\nolinkurl{matija.bucic@math.ethz.ch}}.}
\and
Benny Sudakov\thanks{Department of Mathematics, ETH, Z\"urich, Switzerland. Email:
\href{mailto:benjamin.sudakov@math.ethz.ch} {\nolinkurl{benjamin.sudakov@math.ethz.ch}}.
Research supported in part by SNSF grant 200021-175573.}
}

\maketitle

\section{Introduction}

Given a fixed graph $H$, over all $n$-vertex graphs $G$ what is the maximum number of $2$-edge colourings of $G$ which contain no monochromatic copy of $H$? We denote the answer by $F(n,H)$. This very natural question was first asked by Erd\H{o}s and Rothschild \cite{E-R} in 1984 for the special case of $H=K_3$. This case was resolved by Yuster \cite{yuster} for large $n$ who in turn raised the problem of determining $F(n,K_k)$. This problem, again for large $n$, was solved by Alon, Balogh, Keevash and Sudakov \cite{ABKS} who in addition solve it for $H$ being any edge-colour critical graph (defined as graphs in which removal of some edge results in decrease in the chromatic number). This question has attracted a lot of attention over the years and has been generalised in a number of ways, we point the interested reader to the numerous papers citing \cite{ABKS}.

In 2006 Alon and Yuster \cite{alon-yuster} raised a closely related problem of maximising, over all $n$-vertex graphs $G$, the number of orientations of $G$ which contain no copy of some fixed tournament $T$. We denote the answer by $D(n,T)$. An immediate lower bound comes from taking $G$ to be a $K_{k}$-free graph with maximum number of edges. Since any orientation of such a $G$ is free of any tournaments on $k$ vertices this gives $D(n,T) \ge 2^{|E(G)|}=2^{t_{k-1}(n)}$ for any $k$-vertex tournament $T$, where $t_{k-1}(n)$ denotes the Tur\'an number. Alon and Yuster \cite{alon-yuster} show that, for large $n$, this easy lower is in fact the answer. Their general argument, which follows the approach used in \cite{ABKS}, relies on a regularity lemma and hence results in a requirement for $n$ to be extremely large. For the special case of $3$ vertex tournaments they give a different approach which solves the problem for the transitive tournament on $3$ vertices for all $n$ and only requires $n$ to be larger than about $10^4$ for $C_3$. Recently Araujo, Botler and Mota \cite{mota} determine the answer for $C_3$ for all values of $n$.

Araujo, Botler and Mota \cite{mota} raise a very natural question of what happens if instead of tournaments we are interested in $D(n,H)$ for an arbitrary oriented graph. In particular, they single out the question of what happens if $H$ is a strongly connected directed cycle $C_{k}$, even if we are only interested in the case of large $n$. In this short remark we answer their question for odd cycles. 

\begin{thm}\label{thm:main}
    For any $k \ge 1$ there exists $n_0=n_0(k)$ such that if $n \ge n_0$ $$D(n,C_{2k+1})=2^{\floor{n^2/4}}.$$
\end{thm}

In fact, our argument applies for any $H$ which is an orientation of an edge-colour critical graph. Since our argument follows closely the ideas of both \cite{ABKS,alon-yuster} 
we will only give a short sketch and only for the case of odd cycles. Another benefit of this approach is that we are able to present the key ideas behind all of these arguments without burying them under the details as tends to happen when making regularity based proofs formal.

\section{Proof sketch}
We refer the reader to \cite{alon-yuster} for how to fill in various details and computations since their results are in the same setting as we are working in. We will assume some familiarity with the basic directed regularity lemma, the specific details needed are given in Section 2 of \cite{alon-yuster}.

The following lemma says that if there are many orientations of $G$ which are $C_{2k+1}$-free then $G$ is not far from being bipartite. It is analogous to Lemma 2.1 in \cite{alon-yuster} which replaces $C_{2k+1}$ with an arbitrary tournament (and adjusts the numbers accordingly). 

\begin{lem}\label{lem:stability}
    Let $k \ge 1$ and $\delta>0$ there exists $n_0=n_0(\delta,k)$ such that if $G$ is a graph of order $n \ge n_0$ which has at least $2^{\floor{n^2/4}}$ distinct $C_{2k+1}$-free orientations then there is a bipartition of $V(G)$ with at most $\delta n^2$ edges inside parts.
\end{lem}
    
\begin{proof}[ (sketch)]
Let $\overrightarrow{G}$ be a $C_{2k+1}$-free orientation of $G$. We apply directed regularity lemma to $\overrightarrow{G}$ to obtain an $\eps$-regular partition $V(\overrightarrow{G})=V_1 \cup \ldots \cup V_m$ (all $V_i$'s should have sizes as equal as possible, and all but $\eps m^2$ pairs $(V_i, V_j)$ should satisfy that linear sized subsets have about the same density of edges in both directions as the density between $V_i,V_j$). We then consider a cluster graph of density $\eta$ (so vertices being parts of our partition and two parts joined by a directed edge if they are $\eps$-regular and the density of edges in the corresponding direction is at least $\eta$).

We first want to show that there must exist some orientation $\overrightarrow{G}$ for which the resulting cluster graph has at least $m^2/4-\beta m^2$ edges directed both ways. If this is not the case there would be too few (less than $2^{n^2/4}$) orientations possible. Indeed, we will fix a partition $\mathcal{P}$ and a cluster graph $C$ and count how many orientations could result with this partition and the cluster graph. Since there are relatively few (recall that regularity lemma gives us a constant number of parts) possible partitions and cluster graphs and every orientation gives rise to some partition and cluster graph this will result in too few orientations in total. There are few edges inside parts of our fixed $\mathcal{P}$ and between non $\eps$-regular pairs (at most $\eps n^2$ in both cases) and each edge may be oriented in two ways so total contribution of such edges is at most a factor of $2^{2\eps n^2}$ to the number of orientations. Similarly, for any $\eps$-regular pair $(V_i,V_j)$ which is not an edge of $C$ in some direction there must be at most about $\eta n^2/m^2$ directed edges, so a large proportion of the edges are oriented the same way. An easy estimate tells us that edges can be oriented this way in at most $2^{c_{\eta}n^2/m^2}$ many ways where $c_{\eta}$ is a small constant depending on $\eta$. Since there are at most $m^2$ such pairs, orienting edges between them contributes at most a factor of $2^{c_{\eta}n^2}$ to the total number of orientations. Finally, for any edge of $C$ there are at most $2^{(n/m)^2}$ orientations of edges between the corresponding pair, but since we are assuming $C$ has at most $m^2/4-\beta m^2$ edges they contribute at most a factor of $2^{n^2/4-\beta n^2}$ to the total number of orientations. Choosing $\eta$ to be small enough compared to $\beta$ gives us a contradiction to having at least $2^{\floor{n^2/4}}$ orientations.

Let now $\overrightarrow{G}$ be an orientation for which the resulting cluster graph $C$ has at least $m^2/4-\beta m^2$ edges directed both ways. We claim $C$ can not contain a bidirected triangle missing only a single directed edge as otherwise $\overrightarrow{G}$ would contain a $C_{2k+1}$. This is a consequence of a standard embedding lemma along the lines of Lemma 2.5 of \cite{alon-yuster} and can be deduced from it by refining the partition (splitting each part into $k$ parts, while preserving the regularity and density with somewhat worse constants) and then using their lemma to embed $C_{2k+1}$ one vertex per new part.

In particular, this tells us that the graph consisting only of bidirected edges of $C$ is both triangle free and has at least $m^2/4-\beta m^2$ edges. The stability theorem of Simonovits \cite{simonovits-stability} tells us that there is a bipartition of $V(C)=W_1 \cup W_2$ with at most $\alpha m^2$ bidirected edges within a part (for any $\alpha$, provided $\alpha \gg \beta$). If we consider a bipartite subgraph consisting of bidirected edges of $C$ between $W_1$ and $W_2$ it has at least $m^2/4-(\beta+\alpha) m^2$ edges. In particular, if $C$ had in addition $9(\alpha+\beta)m^2$ directed edges we would find a bidirected triangle with one directed edge removed in $C$. This follows since at least $8(\alpha+\beta)m^2$ of these additional edges must be inside parts so at least $4(\alpha+\beta)m^2$ inside a single part, say $W_1$, and we can pass to a bipartite subgraph of size $2(\alpha+\beta)m^2$ within $W_1$. Taking into account these edges might also be bidirected there are $(\alpha+\beta)m^2$ distinct pairs spanning a directed edge. These edges together with the bidirected edges across make a subgraph with at least $m^2/4$ edges so by Mantel's theorem make a triangle. This triangle needs to have at most one vertex inside the part (since edges inside the part make a bipartite graph) so they make a desired triangle in $C$.

This tells us that after removing edges within $V_i$'s, between non-$\eps$-regular pairs, between pairs having density less than $\eta$ in some direction (since we know there are at most $9(\alpha+\beta)m^2$ such pairs) and edges inside $W_1$ or $W_2$ above (at most $\alpha m^2$ such pairs) we removed at most a small constant proportion of edges and are left with a bipartite graph, as desired. 
\end{proof}

The following lemma replaces the embedding Lemma 3.1 of \cite{alon-yuster}. Let us introduce some notation for convenience.
Given a directed graph $G$ and an integer $k$ we say a pair of disjoint subsets $W_1,W_2\subseteq V(G), |W_i|\ge 2k$ are $k$-regular if for any $X_i \subseteq W_i, |X_i|\ge |W_i|/20$ for $i=1,2$ we have at least $1/10$ proportion of edges of $G$ directed from $X_1$ to $X_2$ as well as from $X_2$ to $X_1.$

\begin{lem}\label{lem:embedding}
    Let $G$ be a directed graph and $W_1,W_2\subseteq V(G)$ a $k$-regular pair. Then one can find a directed path of length $2k$ starting in either of $W_i$.
\end{lem}
    
\begin{proof}
    We iteratively find our directed path. Assume that at stage $2i-1$ we found a path $v_1,\ldots, v_{2i-1}$ and a subset $V_{2i-1}\subseteq W_2 \setminus \{v_2,v_4,\ldots,v_{2i-2}\}$ of at least $|W_2|/20$ out-neighbours of $v_{2i-1}$. Then since there is a 1/10 proportion of edges oriented from $V_{2i-1}$ to $W_1 \setminus \{v_1,v_3,\ldots,v_{2i-1}\}$ there must be a vertex $v_{2i}$ in $V_{2i}$ with a set $V_{2i}$ of at least $(|W_1|-k)/10\ge |W_1|/20$ out-neighbours in $W_1 \setminus\{v_1,v_3,\ldots,v_{2i-1}\}$. Repeating from the other side completes the iteration. After $2k$ iterations we find the desired path. 
\end{proof}

We now turn to the proof of our main result.

\begin{proof}[ of \Cref{thm:main}]
    Let $n_0$ be given by \Cref{lem:stability} with sufficiently small $\delta$.
    
    Let us take a graph $G$ on $n>n_0^2+n_0$ vertices which has at least $2^{\floor{n^2/4}+m}$ $C_{2k+1}$-free orientations for some $m \ge 0$. We will show that if $G$ is not the Tur\'an graph then we can find a vertex $v$ such that $G\setminus v$  has at least $2^{\floor{(n-1)^2/4}+m+1}$ $C_{2k+1}$-free orientations. We then iterate (note that no subgraph we consider can any longer be a Tur\'an graph since it has too many orientations, so also edges) as long as our graph has at least $n_0$ vertices. When we stop we obtain a graph with less than $n_0$ vertices which has at least $2^{n_0^2}$ orientations which is impossible since it has at most $n_0^2/2$ edges. Let us assume $G$ is not the Tur\'an graph on $n \ge n_0$ vertices and proceed to find such a vertex $v$.
    
    Every vertex needs to have degree at least $\floor{n/2}$ as otherwise its edges contribute at most a factor of $2^{\floor{n/2}-1}$ to the number of orientations so it would work as our vertex $v$ above.
    
    Let $V_1,V_2$ make a bipartition of $V(G)$ which minimises the number of edges within parts. 
    Since $n \ge n_0$ by \Cref{lem:stability} we have few (in particular at most $\delta n^2$) edges within parts. This implies $|V_1|,|V_2| \le (1/2+\delta^{1/2})n$ as otherwise there would be less than $n^2/4$ edges, so too few orientations. Similarly, there can be at most $\delta n^2$ edges missing between parts.
    
    We first claim that there can be only few orientations for which there exists a pair of subsets  $X_1\subseteq V_1, X_2\subseteq V_2$, both of size at least $2\delta n$, which have at most 1/10 proportion of edges directed from $X_1$ to $X_2$. The number of such orientations of edges between $X_1,X_2$ is at most $\binom{e(X_1,X_2)}{\le e(X_1,X_2)/10}\le 2^{0.5e(X_1,X_2)}$. Since the total number of edges is at most $n^2/4+\delta n^2$ there are at most $2^{n^2/4+\delta n^2-0.5e(X_1,X_2)}$ such orientations of the whole graph. Since we are missing at most $\delta n^2$ edges between $V_1,V_2$ we have $e(X_1,X_2)\ge |X_1||X_2|-\delta n^2 \ge 3 \delta n^2$ the number of such orientations is at most $2^{n^2/4-0.5\delta n^2}$. Since we can choose possible locations of $X_1$ and $X_2$ in at most $2^{2n}$ many ways there can be at most $2^{2n}\cdot 2^{n^2/4-0.5\delta n^2}\le 2^{\floor{n^2/4}}/2$ orientations for which such a pair $X_1,X_2$ exists.
        
    Let us now consider only $C_{2k+1}$-free orientations for which any pair of subsets $X_1,X_2$ of size at least $2\delta n$ have at least $1/10$ proportion of edges oriented in both ways. In particular, any pair of subsets both of size at least $40\delta n$ is $k$-regular. We call such an orientation relevant and by above counting there are at least $2^{\floor{n^2/4}+m}-2^{\floor{n^2/4}}/2\ge 2^{\floor{n^2/4}+m-1}$ relevant orientations.
    
    \textbf{Case 1.} Some vertex $v$ has at least $800\delta n$ neighbours in its own part, say $V_1$.
    
    Note that $v$ must have at least $800 \delta n$ neighbours in $V_2$ as well (by us taking the max-cut). If in a relevant orientation $v$ has $40\delta n$ out-neighbours and $40\delta n$ in-neighbours belonging to different parts then since these sets make a $k$-regular pair we can find a path of length $2k-1$ and join it with $v$ to find a $C_{2k+1}$, a contradiction. This implies that $v$ must have at most $80\delta n$ in neighbours or at most $80\delta n$ out-neighbours. In particular, its edges may be oriented in such a way in at most $2\cdot \binom{d(v)}{\le 80 \delta n}\le 2\binom{d(v)}{\le d(v)/10} \le 2^{0.49d(v)}$ many ways.
    
    In other words $G \setminus \{v\}$ must have at least $2^{\floor{n^2/4}+m-1-0.49d(v)}\ge 2^{\floor{(n-1)^2/4}+m+1}$ $C_{2k+1}$-free orientations (since $d(v) \le n$ and $n$ is large) as desired.
    
    \textbf{Case 2.} Every vertex of $G$ has at most $800\delta n$ neighbours in its own part.
    
    Since $G$ is not the Tur\'an graph there must exist an edge $uv$ inside a part. Both $u$ and $v$ have at least $\floor{n/2}-800\delta n \ge n/3$ neighbours in the other part, in particular they have $d(u,v) \ge n/8$ common neighbours since parts have size at most $n/2+\delta^{1/2}n$. If in a relevant orientation $uv$ is an edge then there can be at most $40 \delta n$ out-neighbours of $v$ which are also in-neighbours of $u$ in the other part, as otherwise \Cref{lem:embedding} provides us with $C_{2k+1}$. This will severely reduce the number of possible orientations of edges incident to $u,v.$ In particular, the edges from $u$ and $v$ to their common neighbours can be oriented in at most $\binom{d(u,v)}{40 \delta n} \cdot 4^{40\delta n} \cdot 3^{d(u,v)-40\delta n}\le 4^{0.99d(u,v)}.$ The same bound analogously holds if $vu$ is the edge instead. In particular, there are at most $2^{d(u)+d(v)-0.02d(u,v)}\le 2^{n-n/1000}$ possible orientations of edges incident to $u$ and $v$ (we are using that both $u$ and $v$ have degree at most $n/2+ \delta^{1/2}n+800\delta n$). In particular, the total number of orientations of $G \setminus \{u,v\}$ is going to be at least $2^{\floor{(n-2)^2/4}+m+2}$ so we made two steps of our argument at once and are done.
\end{proof}
    
\providecommand{\bysame}{\leavevmode\hbox to3em{\hrulefill}\thinspace}
\providecommand{\MR}{\relax\ifhmode\unskip\space\fi MR }
\providecommand{\MRhref}[2]{%
  \href{http://www.ams.org/mathscinet-getitem?mr=#1}{#2}
}
\providecommand{\href}[2]{#2}

\end{document}